\documentclass[reqno,12pt]{amsart}
\usepackage{amsmath, epsfig, cite}
\usepackage{amssymb}
\usepackage{amsfonts}
\usepackage{latexsym}
\usepackage{amsthm}
\usepackage [latin1]{inputenc}

\newtheorem{theorem}{Theorem}%[section]

\newtheorem{lemma}{Lemma}

\theoremstyle{remark}

\numberwithin{equation}{section}

\allowdisplaybreaks

\author{Guoping Gu}
\address{DEPARTMENT OF MATHEMATICS, SHANGHAI UNIVERSITY, SHANGHAI 200444, P. R. CHINA}
\email{$^*$ Corresponding author. guguoping0@163.com (G. Gu), xiaoxiawang@shu.edu.cn (X. Wang).}

\author{Xiaoxia Wang$^*$}
\title{some double series for $\pi$ and their $q$-analogues}


\begin{document}
	
	\begin{abstract}
		In this paper, by applying the partial derivative operator on summation formulas of hypergeometric series and basic hypergeometric series, we establish several double series for $\pi$ and their q-analogues.
	\end{abstract}
	
	\maketitle
	
	%%%%%%%%%%%%%%%%%%%%%%%%%%%%%%%%%%%%%%%%%%%%%%%%%%%%%%%%%%%%%%%%%%%%%%%%%%%%%%%%%%%%%%%%%%%%%%%%%%%%%%%%%%%%%%%%%%%%%%%%%%%%%%%%%%%%%%%%%%%%%%%%%%%%%%%%%%%%%%%%%%%%%%%%%%%%%%%%%%%%%%%%%%%%%%%%%%%%%%%%%%%%%%%%%%%%%%%%%%%%%%%%%%%%%%%%%%%%%%%%%%%%%%%%%%%%%%%%%%%%%%%%%%%%%%%%%%%%%%%%%%%%%%%%%%%%%%%%%%%%%%%%%%%%%%%%%%%%%%%%%%%%%%%%%%%%%%%%%%%%%%%%%%%%%%%%%%%%%%%%%%%%%%%%%%%%%%%%%%%%%%%%%%%%%%%%%%%%%%%%%%%%%%%%%%%%%%%%%%%%%%%%%%%%%%%%%%%%%%%%%%%%%%%%%%%%%%%%%%%%%%%%%%%%%%%%%%%%%%%%

	\section{Introduction}
	
	In 1914, Ramanujan \cite{Ramanujan} laid out 17 series for $1/\pi$ without proof, and then all of them were proved by Borweins \cite{Borweins}. One of Ramanujan's formulas is stated as
	\begin{equation*}
		\sum_{k=0}^{\infty}(6k+1)\frac{(\frac12)_k^3}{k!^34^k}=\frac{4}{\pi},
	\end{equation*}
	where the shifted factorial $(x)_n$ is defined as follows: for a complex number $x$,
	\begin{equation*}
		(x)_n=\begin{cases}
			x(x+1)(x+2)\cdots(x+n-1)& \text{if} \quad n\in\mathbb{Z}^+;\\
			1&\text{if}\quad n=0.
		\end{cases}
	\end{equation*}
	In addition, $(x)_n$ can also be expressed by
	\begin{equation*}
		(x)_n=\frac{\Gamma(x+n)}{\Gamma(x)}.
	\end{equation*}
	Here $\Gamma(x)$ stands for the well-known Gamma function which is defined as
	\begin{equation*}
		\Gamma(x)=\int_{0}^{\infty}t^{x-1}e^{-t}dt\qquad\text{with\quad $Re(x)>0$}.
	\end{equation*}
	The Gamma function has many important properties, for example:
	\begin{equation*}
		\Gamma(x+1)=x\Gamma(x),\quad\Gamma(x)\Gamma(1-x)=\frac{\pi}{\sin(\pi x)},\quad\lim_{n\rightarrow\infty}\frac{\Gamma(x+n)}{\Gamma(y+n)}n^{y-x}=1,
	\end{equation*}
	which will be used without explanation in the following of this paper.
	
	Recently, Wei \cite{Wei-1} certified the following two double series for $\pi$ which were conjectured by Guo and Lian \cite{Lian}:
	\begin{align}
		\sum_{k=1}^{\infty}(6k+1)\frac{(\frac12)_k^3}{k!^34^k}\sum_{j=1}^{k}\left(\frac{1}{(2j-1)^2}-\frac{1}{16j^2}\right)&=\frac{\pi}{12},&\nonumber\\
		\sum_{k=1}^{\infty}(-1)^k(6k+1)\frac{(\frac12)_k^3}{k!^38^k}\sum_{j=1}^{k}\left(\frac{1}{(2j-1)^2}-\frac{1}{16j^2}\right)&=-\frac{\sqrt{2}\pi}{48}\label{eq:1-1}.
	\end{align}
	Besides, Swisher \cite{Swisher} gave a congruence on truncated form of  \eqref{eq:1-1}, which was conjectured by Long \cite{Long} in 2011.
	Furthermore, Wei also similarly gives another three double series for $\pi$ in \cite{Wei-2} and one of them is
	\begin{equation}\label{eq:1-2}
		\sum_{k=1}^{\infty}(-1)^k(4k+1)\frac{(\frac12)_k^3}{k!^3}\sum_{i=1}^{2k}\frac{(-1)^i}{i^2}=\frac{\pi}{12},
	\end{equation}
	with the $q$-analogue as
	\begin{equation}
		\sum_{k=1}^{\infty}(-1)^kq^{k^2}[4k+1]\frac{(q;q^2)_k^3}{(q^2;q^2)_k^3}\sum_{i=1}^{2k}(-1)^i\frac{q^i}{[i]^2}=\frac{(q,q^3;q^2)_{\infty}}{(q^2;q^2)^2_{\infty}}\sum_{j=1}^{\infty}\frac{q^{2j}}{[2j]^2} \label{eq:1-3}.
	\end{equation}
	Here the $q$-integer $[n]$ is defined as $[n]=(1-q^n)/(1-q)=1+q+\cdots+q^{n-1}$ and the $q$-shifted factorial is stated as
	\begin{equation*}
		(x;q)_{\infty}=\prod_{i=0}^{\infty}(1-xq^i),~(x;q)_n=\frac{(x;q)_{\infty}}{(xq^n;q)_{\infty}}.
	\end{equation*}
	For convenience, we shall also adopt the following notation
	\begin{equation*}
		(x_1,x_2,\cdots,x_r;q)_m=(x_1;q)_m(x_2;q)_m\cdots(x_r;q)_m,  \quad m\in\mathbb{Z}^+\cup\{0,\infty\}.
	\end{equation*}
	
	Inspired by the work just mentioned, we deduce some double series for $\pi$ and their $q$-analogues by applying the partial derivative operator.
	For more known series on $\pi$, we refer the reader to the papers \cite{Chan-1,Chan-2,Liu,Sun,Zudilin}.
	
	For the completeness of this paper, we will introduce the following concept. For a multivariable function $f(x_1,x_2,\cdots,x_m)$, define the partial derivative operator $\mathcal{D}_{x_i}$ by
	\begin{equation*}
		\mathcal{D}_{x_i}f(x_1,x_2,\cdots,x_m)=\frac{d}{dx_i}f(x_1,x_2,\cdots,x_m) \qquad \text{with $1\leq i\leq m$}.
	\end{equation*}

	\section{Main results}
	Applying the partial derivative operator on the summation formula of hypergeometric series, we get the following double series for $\pi$.
	\begin{theorem}\label{thm:1}
		The following result is true.
		\begin{equation}\label{eq:2-1}
			\sum_{k=1}^{\infty}(6k-1)\frac{(-\frac12)^2_k}{4^kk!(\frac32)_k}\sum_{i=1}^{k}\frac{1}{(2i-1)^2}=\frac{\pi^3}{144}.
		\end{equation}
	\end{theorem}
	With the help of Gasper and Rahman's quadratic summation formula \cite[Equation~(3.8.12)]{G-R}, we get the $q$-analogue of \eqref{eq:2-1}.
	\begin{theorem}\label{thm:2}
		Let the complex number $q$ with $0<|q|<1$. Then, we have
		\begin{align}
			\sum_{k=1}^{\infty}&[6k-1]\frac{(q^{-1},q,q;q^2)_k(q^{-2};q^4)_k}{(q^4,q^2,q^2;q^4)_k(q^{3};q^2)_k}q^{(k+1)^2}\sum_{i=1}^{k}\left(\frac{q^{2i-1}}{[2i-1]^2}-\frac{q^{4i-2}}{[4i-2]^2}\right)&\nonumber\\
			=&\frac{(q,q^4,q^4;q^4)_{\infty}}{(q^5,q^2,q^2;q^4)_{\infty}}\sum_{i=1}^{\infty}\left(\frac{q^{4i-2}}{[4i-2]^2}-\frac{q^{4i}}{[4i]^2}\right).
		\end{align}
	\end{theorem}
	In addition, we establish many other double series for $\pi$ also from the summation formula of hypergeometric series.
	\begin{theorem}\label{thm:3}
		The following results are true.
		\begin{flalign}
			&\sum_{k=1}^{\infty}(4k+1)\frac{(-\frac12)_k(\frac12)_k^3}{(k+1)!k!^3}\sum_{i=1}^{2k}(-1)^{i-1}\frac{1}{i^2}=\frac23-\frac{8}{\pi^2},\\
			&\sum_{k=1}^{\infty}(4k+3)\frac{(-\frac12)_k(\frac12)_k^2(\frac32)_k}{k!(k+1)!^2(k+2)!}\sum_{i=1}^{k}\left(\frac{1}{(2i-1)^2}-\frac{1}{4(i+1)^2}\right)=\frac{32}{27}-\frac{992}{81\pi^2},\\
			&\sum_{k=1}^{\infty}(4k+3)\frac{(\frac32)_k(\frac12)_k^3}{k!(k+1)!^3}\sum_{i=1}^{k}\left(\frac{1}{(2i-1)^2}-\frac{1}{4(i+1)^2}\right)=\frac{8}{3}-\frac{24}{\pi^2},\\
			&\sum_{k=1}^{\infty}(-1)^k(4k+3)\frac{(\frac32)_k(\frac12)_k^2}{k!(k+1)!^2}\sum_{i=1}^{k}\left(\frac{1}{(1+i)^2}-\frac{4}{(2i-1)^2}\right)=\frac{4\pi}{3}-\frac{8}{\pi}.
		\end{flalign}
	\end{theorem}
	The $q$-analogues of the double series for $\pi$ in Theorem \ref{thm:3} are expressed as follows.
	\begin{theorem}\label{thm:4}
		Let the complex number $q$ with $0<|q|<1$. Then, we have
		\begin{align}
			&\sum_{k=1}^{\infty}[4k+1] \frac{(q;q^2)^3_k(q^{-1};q^2)_k}{(q^2;q^2)^3_k(q^4;q^2)_k}q^{2k}
			\sum_{i=1}^{k}\left(\frac{q^{2i}}{[2i]^2}-\frac{q^{2i-1}}{[2i-1]^2}\right)&\nonumber\\
			&\qquad\qquad\qquad=\frac{(q^3;q^2)^3_{\infty}(q;q^2)_{\infty}}{(q^2;q^2)^3_{\infty}(q^4;q^2)_{\infty}}\sum_{i=1}^{\infty}\left(\frac{(-1)^{i+1}}{[i+1]^2}-\frac{q^{3i}(1-q)}{2[2i+1]^2}\right)q^{i+1},\\
			&\sum_{k=1}^{\infty}[4k+3] \frac{(q;q^2)^2_k(q^3,q^{-1};q^2)_k}{(q^4;q^2)^2_k(q^2,q^6;q^2)_k}q^{4k}
			\sum_{i=1}^{k}\left(\frac{q^{2i+2}}{[2i+2]^2}-\frac{q^{2i-1}}{[2i-1]^2}\right)&\nonumber\\		&\qquad\qquad\qquad=\frac{(q^3;q^2)_{\infty}(q^5;q^2)^2_{\infty}(q^3;q^2)_{\infty}}{(1-q)(q^4;q^2)^3_{\infty}(q^6;q^2)_{\infty}}\sum_{i=1}^{\infty}\left(\frac{(-1)^{i+1}}{[i+3]^2}-\frac{q^{3i+2}(1-q)}{2[2i+3]^2}\right)q^{i+3},\\
			&\sum_{k=1}^{\infty}[4k+3] \frac{(q;q^2)^3_k(q^3;q^2)_k}{(q^4;q^2)^3_k(q^2;q^2)_k}q^{2k}
			\sum_{i=1}^{k}\left(\frac{q^{2i+2}}{[2i+2]^2}-\frac{q^{2i-1}}{[2i-1]^2}\right)&\nonumber\\
			&\qquad\qquad\qquad=\frac{(q^3;q^2)^3_{\infty}(q^3;q^2)_{\infty}}{(1-q)(q^4;q^2)^3_{\infty}(q^2;q^2)_{\infty}}\sum_{i=1}^{\infty}\left(\frac{(-1)^{i}}{[i+2]^2}+\frac{q^{3i}(1-q)}{2[2i+1]^2}\right)q^{i+2},\\
			&\sum_{k=1}^{\infty}(-1)^k[4k+3]\frac{(q;q^2)_{k+1}(q;q^2)^2_k}{(q^2,q^4,q^4;q^2)_k}q^{-k(k+4)}\sum_{i=1}^{k}\left(\frac{q^{2i+2}}{[2i+2]^2}-\frac{q^{2i-1}}{[2i-1]^2}\right)&\nonumber\\
			&\qquad\qquad\qquad=\frac{(q^3,q^3;q^2)_{\infty}}{(q^4,q^4;q^2)_{\infty}}\sum_{i=1}^{\infty}\frac{q^{2i+2}}{[2i+2]^2}.
		\end{align}			
	\end{theorem}
	
	The structure of this paper is arranged as follows. We shall prove Theorem \ref{thm:1} in Section 3 and Theorem \ref{thm:2} in Section 4 by the partial derivative operator respectively. Finally, we will certify Theorems \ref{thm:3} and \ref{thm:4} in Section 5.
	
	\section{Proof of Theorem \ref{thm:1}}
	Following Gasper and Rahman \cite{G-R}, the hypergeometric series is defined as
	\begin{equation*}
		\begin{split}
			_{r+1}F_r\left[
			\begin{array}{cccc}
				a_1,&a_2,&\dots,&a_{r+1}\\
				b_1,&b_2,&\dots,&b_r
			\end{array};z
			\right]=\sum_{k=0}^{\infty}\frac{(a_1)_k(a_2)_k \dots (a_{r+1})_k}{(b_1)_k(b_2)_k\dots(b_r)_k}\frac{z^k}{k!}
		\end{split}.
	\end{equation*}
	In order to prove Theorem \ref{thm:1}, we recall the following $_7F_6$-summation formula\cite[Equation~(1.7)]{G-D}:
	\begin{equation}\label{eq:3-1}
		\begin{split}
			_7F_6
			&\left[
			\begin{array}{ccccccc}
				a,&1+\frac a3,&b,&1-b,&c,&\frac12+a-c+n,&-n\\[3pt]
				&\frac a3,&\frac{2+a-b}{2},&\frac{1+a+b}{2},&1+a+2n,&1+a-2c,&2c-a-2n
			\end{array}
			\right]\\[3pt]
			&=\left[
			\begin{array}{cccc}
				\frac{1+a}{2},&1+\frac a2,&\frac{1+a+b}{2}-c,&1+\frac{a-b}{2}-c\\[5pt]
				\frac{1+a+b}{2},&1+\frac{a-b}{2},&\frac{1+a}{2}-c,&1+\frac a2-c
			\end{array}
			\right]_n.
		\end{split}
	\end{equation}	
	
	\begin{proof}[Proof of Theorem \ref{thm:1}]
		Applying the partial derivative $\mathcal{D}_b$ on both sides of the $_7F_6$-summation formula \eqref{eq:3-1}, we subsequently obtain the following result
		\begin{equation}
			\begin{split}
				&\sum_{k=1}^{n}\frac{(a)_k(1+\frac a3)_k(b)_k(1-b)_k(c)_k(\frac 12+a-c+n)_k(-n)_k}{(1)_k(\frac a3)_k(\frac{2+a-b}{2})_k(\frac{1+a+b}{2})_k(1+a+2n)_k(1+a-2c)_k(2c-a-2n)_k} A_k(a,b)\\[2mm]
				&=\frac{(\frac{1+a}{2})_n(1+\frac a2)_n(\frac{1+a+b}{2}-c)_n(1+\frac{a-b}{2}-c)_n}{(\frac{1+a+b}{2})_n(1+\frac{a-b}{2})_n(\frac{1+a}{2}-c)_n(1+\frac a2-c)_n} B_n(a,b,c).\label{eq:3-2}
			\end{split}
		\end{equation}
		where
		\begin{align*}
			A_k(a,b)&=H_k(b-1)-H_k(-b)+\tfrac12H_k(\tfrac{a-b}{2})-\tfrac12H_k(\tfrac{a+b-1}{2}),\\[2mm]
			B_n(a,b,c)&=\tfrac12H_n(\tfrac{a+b-1}{2}-c)-\tfrac12H_n(\tfrac{a-b}{2}-c)-\tfrac12H_n(\tfrac{a+b-1}{2})+\tfrac12H_n(\tfrac{a-b}{2}).
		\end{align*}
		Here $H_k(x)$ represents the generalized harmonic number, which is defined by
		\begin{equation*}
			H_k(x)=\sum_{i=1}^{k}\frac{1}{x+i}.
		\end{equation*}
		More generally, $H_k^{(m)}(x)$ represents the generalized harmonic number of order $m$ as
		\begin{equation*}
			H_k^{(m)}(x)=\sum_{i=1}^{k}\frac{1}{(x+i)^m}.
		\end{equation*}
		Taking $m=1$ and $x=0$ in $H_k^{(m)}(x)$, we get the famous classical harmonic number.

		Now employing the partial derivative operator $\mathcal{D}_b$ again on both sides of \eqref{eq:3-2}, we get
		\begin{equation}
			\begin{split}
				&\sum_{k=1}^{n}\frac{(a)_k(1+\frac a3)_k(b)_k(1-b)_k(c)_k(\frac 12+a-c+n)_k(-n)_k}{(1)_k(\frac a3)_k(\frac{2+a-b}{2})_k(\frac{1+a+b}{2})_k(1+a+2n)_k(1+a-2c)_k(2c-a-2n)_k}\\[2mm]
				&\quad\times\{A_k(a,b)^2+C_k(a,b)\}\\[2mm]
				&=\frac{(\frac{1+a}{2})_n(1+\frac a2)_n(\frac{1+a+b}{2}-c)_n(1+\frac{a-b}{2}-c)_n}{(\frac{1+a+b}{2})_n(1+\frac{a-b}{2})_n(\frac{1+a}{2}-c)_n(1+\frac a2-c)_n}\{B_n(a,b,c)^2+D_n(a,b,c)\}.\label{eq:3-3}
			\end{split}
		\end{equation}
		where
		\begin{align*}
			C_k(a,b)=&-H^{(2)}_k(b-1)-H^{(2)}_k(-b)+\tfrac14H^{(2)}_k(\tfrac{a-b}{2})+\tfrac14H^{(2)}_k(\tfrac{a+b-1}{2}),\\[2mm]
			D_n(a,b,c)=&-\tfrac14H^{(2)}_n(\tfrac{a+b-1}{2}-c)-\tfrac14H^{(2)}_n(\tfrac{a-b}{2}-c)+\tfrac14H^{(2)}_n(\tfrac{a+b-1}{2})+\tfrac14H^{(2)}_n(\tfrac{a-b}{2}).
		\end{align*}
		
		Note that $A_k(a,b)=B_n(a,b,c)=0$ when $a=-\frac12,b=\frac12,c\rightarrow -\frac12$ in \eqref{eq:3-3}. Then, we subsequently arrive at Theorem \ref{thm:1} by performing the replacements $a=-\frac12,b=\frac12,c\rightarrow -\frac12$ and $n\rightarrow\infty$ in \eqref{eq:3-3}. Here we have used the following property of Gamma function
		\begin{equation*}
			\lim_{n\rightarrow\infty}\frac{\Gamma(x+n)}{\Gamma(y+n)}n^{y-x}=1,
		\end{equation*}
		and Euler's formula	
		\begin{equation*}
			\sum_{i=1}^{\infty}\frac{1}{i^2}=\frac{\pi^2}{6}.
		\end{equation*}
	\end{proof}

	\section{Proof of Theorem \ref{thm:2}}
	Following Gasper and Rahman \cite{G-R}, the basic hypergeometric series $_{r+1}\phi_r$ is defined as
	\begin{equation*}
		\begin{split}
			_{r+1}\phi_r\left[
			\begin{array}{cccc}
				a_1,&a_2,&\dots,&a_{r+1}\\
				b_1,&b_2,&\dots,&b_r
			\end{array};q,z
			\right]=\sum_{k=0}^{\infty}\frac{(a_1,a_2,\cdots,a_{r+1};q)_k}{(q,b_1,b_2,\dots,b_r;q)_k}z^k.
		\end{split}
	\end{equation*}
	For the purpose of proving Theorem \ref{thm:2}, we first present Gasper and Rahman's quadratic summation formula \cite[Equation~(3.8.12)]{G-R} as follows:
	\begin{equation}
		\begin{split}
			\sum_{k=0}^{\infty}&\frac{1-aq^{3k}}{1-a}\frac{(a,b,q/b;q)_k(d,f,a^2q/df;q^2)_k}{(q^2,aq^2/b,abq;q^2)_k(aq/d,aq/f,df/a;q)_k}q^k\\
			&\qquad+\frac{(aq,f/a,b,q/b;q)_{\infty}(d,a^2q/df,fq^2/d,df^2q/a^2;q^2)_{\infty}}{(a/f,fq/a,aq/d,df/a;q)_{\infty}(aq^2/b,abq,fq/ab,bf/a;q^2)_{\infty}}\\
			&\qquad\qquad\times\thinspace_3\phi_2
			\left[\begin{array}{cccc}
				f,bf/a,fq/ab\\
				fq^2/d,df^2q/a^2
			\end{array}	;q^2,q^2
			\right]\\[2mm]
			=&\frac{(aq,f/a;q)_{\infty}(aq^2/bd,abq/d,bdf/a,dfq/ab;q^2)_{\infty}}{(aq/d,df/a;q)_{\infty}(aq^2/b,abq,bf/a,fq/ab;q^2)_{\infty}}.\label{eq:4-1}
		\end{split}
	\end{equation}
	\begin{proof}[Proof of Theorem \ref{thm:2}]
		Performing the replacements $d\rightarrow q^{-2n}$ and $f\rightarrow c^2$ in \eqref{eq:4-1}, we get its truncated form as
		\begin{equation*}
			\begin{split}
				&\sum_{k=0}^{n}\frac{1-aq^{3k}}{1-a}\frac{(a,b,q/b;q)_k(q^{-2n},c^2,a^2q^{2n+1}/c^2;q^2)_k}{(q^2,aq^2/b,abq;q^2)_k(aq^{2n+1},aq/c^2,c^2q^{-2n}/a;q)_k}q^k\\[2mm]
				&=\frac{(aq,aq^2,aq^2/bc^2,abq/c^2;q^2)_n}{(aq/c^2,aq^2/c^2,aq^2/b,abq;q^2)_n}.
			\end{split}
		\end{equation*}
		Applying the partial derivative $\mathcal{D}_b$ on both sides of the above identity twice, then we subsequently obtain the following summation
		\begin{equation}
			\begin{split}
				&\sum_{k=1}^{n}\frac{1-aq^{3k}}{1-a}\frac{(a,b,q/b;q)_k(q^{-2n},c^2,a^2q^{2n+1}/c^2;q^2)_k}{(q^2,aq^2/b,abq;q^2)_k(aq^{2n+1},aq/c^2,c^2q^{-2n}/a;q)_k}q^k\{A^2_k(a,b)+C_k(a,b)\}\\[2mm]
				&\quad=\frac{(aq,aq^2,aq^2/bc^2,abq/c^2;q^2)_n}{(aq/c^2,aq^2/c^2,aq^2/b,abq;q^2)_n}\{B^2_n(a,b,c)+D_n(a,b,c)\}.\label{eq:4-2}
			\end{split}
		\end{equation}
		where
		\begin{align*}
			A_k(a,b)&=\sum_{i=1}^{k}\frac{-q^{i-1}}{1-bq^{i-1}}+\sum_{i=1}^{k}\frac{q^i/b^2}{1-q^i/b}-\sum_{i=1}^{k}\frac{aq^{2i}/b^2}{1-aq^{2i}/b}-\sum_{i=1}^{k}\frac{-aq^{2i-1}}{1-abq^{2i-1}},\\
			B_n(a,b,c)&=\sum_{i=1}^{n}\frac{aq^{2i}/b^2c^2}{1-aq^{2i}/bc^2}+\sum_{i=1}^{n}\frac{-aq^{2i-1}/c^2}{1-abq^{2i-1}/c^2}-\sum_{i=1}^{n}\frac{aq^{2i}/b^2}{1-aq^{2i}/b}-\sum_{i=1}^{n}\frac{-aq^{2i-1}}{1-abq^{2i-1}},
		\end{align*}
		\begin{align*}
			C_k(a,b)=&\sum_{i=1}^{k}\frac{-q^{2i-2}}{(1-bq^{i-1})^2}+\sum_{i=1}^{k}\frac{q^i/b^3(q^i/b-2)}{(1-q^i/b)^2}-\sum_{i=1}^{k}\frac{aq^{2i}/b^3(aq^{2i}/b-2)}{(1-aq^{2i}/b)^2}\\
			&\qquad+\sum_{i=1}^{k}\frac{a^2q^{4i-2}}{(1-abq^{2i-1})^2},\\
			D_n(a,b,c)=&\sum_{i=1}^{n}\frac{aq^{2i}/b^3c^2(aq^{2i}/bc^2-2)}{(1-aq^{2i}/bc^2)^2}+\sum_{i=1}^{n}\frac{-a^2q^{4i-2}/c^4}{(1-abq^{2i-1}/c^2)^2}-\sum_{i=1}^{n}\frac{aq^{2i}/b^3(aq^{2i}/b-2)}{(1-aq^{2i}/b)^2}\\
			&\qquad+\sum_{i=1}^{n}\frac{a^2q^{4i-2}}{(1-abq^{2i-1})^2}.
		\end{align*}
		Clearly, $A_k(a,b)=B_n(a,b,c)=0$ when $a= q^{-\frac12} ,b= q^{\frac12} ,c= q^{-\frac12}$. We have the following identity by substituting $a= q^{-\frac12}, b= q^{\frac12}$ and $c= q^{-\frac12}$ in \eqref{eq:4-2}
		\begin{equation*}
			\begin{split}
				\sum_{k=1}^{n} &\frac{1-q^{2k+\frac32}}{1-q^{\frac32}} \frac{(q^{\frac32},q^{\frac12},q^{\frac12},q^{\frac12},q^{n+\frac52},q^{-n};q)_k}{(q,q^2,q^2,q^2,q^{-n},q^{n+\frac52};q)_k}q^k
				\sum_{i=1}^{k}\left(\frac{q^{i}}{(1-q^{i+1})^2}-\frac{q^{i-\frac32}}{(1-q^{i-\frac12})^2}\right)\\[2mm]
				=&\frac{(q^{\frac52},q^{\frac32},q^{\frac32},q^{\frac32};q)_n}{(q^2,q^2,q^2,q;q)_n}\left(\sum_{i=1}^{2n}(-1)^{i}\frac{q^{\frac i2}}{(1-q^{\frac i2+1})^2}+\frac12\sum_{i=1}^{n}\frac{q^{2i}-q^{2i+\frac12}}{(1-q^{i+\frac12})^2}\right).
			\end{split}
		\end{equation*}
		Finally, we get Theorem \ref{thm:2} by replacing $q$ by $q^2$ and letting $n\rightarrow\infty$ in the above identity.
	\end{proof}
	
	\section{Proofs of Theorems \ref{thm:3} and \ref{thm:4}}
	
	Since the double series for $\pi$ in Theorem \ref{thm:3} are just the $q\rightarrow1$ cases of Theorem \ref{thm:4}, the the proof of Theorem \ref{thm:4} is enough. For proving Theorem \ref{thm:4}, we first present the following Lemma which plays a key role in our proof.
	\begin{lemma}\label{lem-1}
		Let the complex number $q$ with $0<|q|<1$. Then, we have
		\begin{equation}
			\begin{split}
				\sum_{k=1}^{n} &\frac{1-aq^{2k}}{1-a} \frac{(a,b,c,q/b,a^2q^n/c,q^{-n};q)_k}{(q,aq/b,aq/c,ab,cq^{1-n}/a,aq^{n+1};q)_k}q^k\{A^2_k(a,b)+C_k(a,b)\}\\[2mm]
				=&\frac{(aq,aq/bc,a,ab/c;q)_n}{(aq/b,aq/c,ab,a/c;q)_n}\{B^2_n(a,b,c)+D_n(a,b,c)\}.\label{eq:5-1}
			\end{split}
		\end{equation}
		where
		\begin{align*}
			A_k(a,b)&=\sum_{i=1}^{k}\frac{-q^{i-1}}{1-bq^{i-1}}+\sum_{i=1}^{k}\frac{q^i/b^2}{1-q^i/b}-\sum_{i=1}^{k}\frac{aq^i/b^2}{1-aq^i/b}+\sum_{i=1}^{k}\frac{aq^{i-1}}{1-abq^{i-1}},\\
			B_n(a,b,c)&=\sum_{i=1}^{n}\frac{aq^i/b^2c}{1-aq^i/bc}+\sum_{i=1}^{n}\frac{-aq^{i-1}/c}{1-abq^{i-1}/c}-\sum_{i=1}^{n}\frac{aq^i/b^2}{1-aq^i/b}+\sum_{i=1}^{n}\frac{aq^{i-1}}{1-abq^{i-1}},
		\end{align*}
		\begin{align*}
			C_k(a,b)=&-\sum_{i=1}^{k}\frac{q^{2i-2}}{(1-bq^{i-1})^2}+\sum_{i=1}^{k}\frac{(q^i/b-2)q^i/b^3}{(1-q^i/b)^2}-\sum_{i=1}^{k}\frac{(aq^i/b-2)aq^i/b^3}{(1-aq^i/b)^2}\\
			&\qquad+\sum_{i=1}^{k}\frac{a^2q^{2i-2}}{(1-abq^{i-1})^2},\\
			D_n(a,b,c)=&\sum_{i=1}^{n}\frac{(aq^i/bc-2)aq^i/b^3c}{(1-aq^i/bc)^2}-\sum_{i=1}^{n}\frac{a^2q^{2i-2}/c}{(1-abq^{i-1}/c)^2}-\sum_{i=1}^{n}\frac{(aq^i/b-2)aq^i/b^3}{(1-aq^i/b)^2}\\
			&\qquad+\sum_{i=1}^{n}\frac{a^2q^{2i-2}}{(1-abq^{i-1})^2}.
		\end{align*}
	\end{lemma}
	
	Obviously, we can derive Wei's result \eqref{eq:1-3} by repalcing $q\rightarrow q^2$ and letting $n \rightarrow \infty$ after setting $(a,b)=(q^{\frac12},q^{\frac12})$ and $c\rightarrow\infty$ in Lemma \ref{lem-1}.
	\begin{proof}
		Recall Jackson's $q$-Dougall-Dixon formula\cite{jackson} which can be stated as follows
		\begin{equation}
			\begin{split}
				&_8\phi_7
				\left[
				\begin{array}{cccccccc}
					a,qa^{\frac12},-qa^{\frac12},b,c,d,e,q^{-n}\\
					a^{\frac12},-a^{\frac12},aq/b,aq/c,aq/d,aq/e,aq^{n+1}
				\end{array};q,q
				\right]\\[2mm]
				&=\frac{(aq,aq/bc,aq/bd,aq/cd;q)_n}{(aq/b,aq/c,aq/d,aq/bcd;q)_n},\quad \text{where $q^{n+1}a^2=bcde$}.\label{eq:5-2}
			\end{split}
		\end{equation}
		Now, putting $d\rightarrow q/b$ and $e\rightarrow a^2q^{n}/c$ in \eqref{eq:5-2}, we arrive at
		\begin{equation*}
			\begin{split}
				\sum_{k=0}^{n} \frac{1-aq^{2k}}{1-a} \frac{(a,b,c,q/b,a^2q^n/c,q^{-n};q)_k}{(q,aq/b,aq/c,ab,cq^{1-n}/a,aq^{n+1};q)_k}q^k=\frac{(aq,aq/bc,a,ab/c;q)_n}{(aq/b,aq/c,ab,a/c;q)_n}.
			\end{split}
		\end{equation*}
		Then we subsequently conclude that the identity \eqref{lem-1} holds by applying the partial derivative $\mathcal{D}_b$ on both sides of the above identity twice.
	\end{proof}

	\begin{proof}[Proof of Theorem \ref{thm:4}]
		We can subsequently get (2.7) by taking $(a,b,c)=(q^{\frac12},q^{\frac12},q^{-\frac12})$, and then replacing $q\rightarrow q^2$ and letting $n \rightarrow \infty$ in Lemma \ref{lem-1}. Similarly, letting $(a,b,c)=(q^{\frac32},q^{\frac12},q^{-\frac12})$, $(a,b,c)=(q^{\frac32},q^{\frac12},q^{\frac12})$, $(a,b)=(q^{\frac32},q^{\frac12})$ and $c\rightarrow\infty$, we are led to the double summation identities (2.8), (2.9) and (2.10) respectively.
	\end{proof}
	In fact, we can also prove the truth of Theorem \ref{thm:3} by the following Dougall's summation formula \cite[Corollary~(3)]{Chu} in the same way as we have done in the proof of Theorem \ref{thm:1}.
	\begin{equation*}
		\begin{split}
			_7F_6
			&\left[
			\begin{array}{ccccccc}
				a,&1+\frac a2,&b,&c,&d,&e,&-n\\[3pt]
				&\frac a2,&1+a-b,&1+a-c,&1+a-d,&1+a-e,&1+a+n
			\end{array}
			\right]\\[3pt]
			=&\left[
			\begin{array}{cccc}
				1+a,&1+a-b-c,&1+a-b-d,&1+a-c-d\\[3pt]
				1+a-b,&1+a-c,&1+a-d,&1+a-b-c-d
			\end{array}
			\right]_n.
		\end{split}
	\end{equation*}

\end{document}